\documentclass[11pt]{article}  
\usepackage{amsmath}
\usepackage{amssymb}
\usepackage{theorem}
\usepackage{euscript}
\usepackage[colorlinks=true,linkcolor=black,citecolor=black]{hyperref}
\usepackage{mathrsfs}

\usepackage{authblk}
\usepackage{blindtext}
\usepackage{hyperref}
\hypersetup{
	colorlinks,
	citecolor=green,
	filecolor=black,
	linkcolor=blue,
	urlcolor=blue
}

\numberwithin{equation}{section}

\def\citep#1#2{\cite[{#1}]{#2}}

\newtheorem{theorem}{Theorem}[section]
\newtheorem{lemma}[theorem]{Lemma}

\numberwithin{equation}{section}
\newtheorem{proposition}[theorem]{Proposition}
\theoremstyle{definition}
 
\theoremstyle{remark}
\newtheorem{remark}[theorem]{Remark}
\theoremstyle{assumption}
\newtheorem{assumption}{Assumption}

\newcommand{\RefSec}[1]{Section~\textup{\ref{#1}}}

\newcommand{\RefThm}[1]{Theorem~\textup{\ref{#1}}}
\newcommand{\RefProp}[1]{Proposition~\textup{\ref{#1}}}
\newcommand{\RefLem}[1]{Lemma~\textup{\ref{#1}}}

\newcommand{\E}{\mathbb{E}}

\newcommand{\bee}{{\boldsymbol{e}}}

\newcommand{\bx}{{\boldsymbol{x}}}

\newcommand{\by}{{\boldsymbol{y}}}

\newcommand{\bz}{{\boldsymbol{z}}}

\newcommand{\brho}{{\boldsymbol{\rho}}}

\newcommand{\bxi}{{\boldsymbol{\xi}}}
\newcommand{\bnu}{{\boldsymbol{\nu}}}

\newcommand{\bvarrho}{{\boldsymbol{\varrho}}}
\newcommand{\bbeta}{{\boldsymbol{\beta}}}

\newcommand{\bgamma}{{\boldsymbol{\gamma}}}

\newcommand{\btau}{{\boldsymbol{\tau}}}
\newcommand{\bzeta}{{\boldsymbol{\zeta}}}

\newcommand{\rd}{{\rm d}} 
\newcommand*{\bigtimes}{\mathop{\raisebox{-.5ex}{\hbox{\huge{$\times$}}}}}

\def\RR{{\mathbb R}}

\def\NN{{\mathbb N}}

\def\CC{{\mathbb C}}
\def\CC{{\mathbb C}}

\def\NN{{\mathbb N}}
\def\RR{{\mathbb R}}

\def\Bb{{\mathcal B}}

\def\Bb{{\mathcal B}}

\def\Ee{{\mathcal E}}

\def\Gg{{\mathcal G}}

\def\Ll{{\mathcal L}}

\def\Ww{{\mathcal W}}

\def\CC{{\mathbb C}}

\def\NN{{\mathbb N}}
\def\RR{{\mathbb R}}

\def\FF{{\mathbb F}}

\def\supp{\operatorname{supp}}
\def\div{\operatorname{div}}
\def\dist{\operatorname{dist}}

\def\mfu{{\mathfrak u }}

\usepackage{color}

\newcommand{\be}{\begin{equation}}
\newcommand{\ee}{\end{equation}}
\newcommand{\beq}{\begin{eqnarray}}
\newcommand{\beqq}{\begin{eqnarray*}}
\newcommand{\eeq}{\end{eqnarray}}
\newcommand{\eeqq}{\end{eqnarray*}}

\definecolor{darkred}{RGB}{139,0,0}
\definecolor{darkgreen}{RGB}{0,100,0}
\definecolor{darkmagenta}{RGB}{180,0,180}
\definecolor{darkblue}{RGB}{0,0,190}

\title{Analyticity of Parametric Elliptic Eigenvalue Problems and Applications to Quasi-Monte Carlo Methods}
\author{Van Kien Nguyen}
 
\affil{Department of Mathematical Analysis, University of Transport and Communications
	\\ No.3 Cau Giay Street, Lang
	Thuong Ward, Dong Da District, Hanoi, Vietnam
\\
Email: kiennv@utc.edu.vn}

\begin{document}
\maketitle

\begin{abstract} In the present paper, we study the analyticity of the leftmost eigenvalue of the linear elliptic partial differential operator with random coefficient and analyze the convergence rate of the quasi-Monte Carlo method for approximation of the expectation of this quantity. The random coefficient is assumed to be  represented by an affine expansion $a_0(\bx)+\sum_{j\in \NN}y_ja_j(\bx)$, where elements of the parameter vector $\by=(y_j)_{j\in \NN}\in U^\infty$ are independent and identically uniformly distributed on $U:=[-\frac{1}{2},\frac{1}{2}]$. Under the assumption $\|\sum_{j\in \NN}\rho_j|a_j|\|_{L_\infty(D)} <\infty$ with some positive sequence $(\rho_j)_{j\in \NN}\in \ell_p(\NN)$ for $p\in (0,1]$ we show that for any $\by\in U^\infty$, the elliptic partial differential operator has a countably infinite number of eigenvalues $(\lambda_j(\by))_{j\in \NN}$ which can be ordered non-decreasingly. Moreover, the spectral gap $\lambda_2(\by)-\lambda_1(\by)$ is uniformly positive in $U^\infty$. From this, we prove the holomorphic extension property of $\lambda_1(\by)$ to a complex domain in $\CC^\infty$ and estimate mixed derivatives of $\lambda_1(\by)$ with respect to the parameters $\by$ by using Cauchy's formula for analytic functions. Based on these bounds we prove the dimension-independent convergence rate of the quasi-Monte Carlo method to approximate the expectation of $\lambda_1(\by)$. 
%
	
	
\end{abstract}

\section{Introduction}
In the last two decades there has been a tremendous growth of interest in uncertainty quantification for physical, biological, or geological models such as groundwater flow, heat transfer, or risk management in financial mathematics. Normally, these models are described by  partial differential equations 
where the input data may be a random variable or a random field. This induces that the derived  quantity of interest will in general also be a random variable or a random field. 
The computational goal is usually to find the expected value or high-order moments of these derived quantities in which  calculation of high-dimensional (or even infinite) integrals is required. Due to its immunity to the dimension of integration, recently, there is a huge interest in treating uncertainty quantification problems by quasi-Monte Carlo method (QMC) such as  
Dick et al. \cite{Dic13,Dic14,Dic16}, Gantner et al. \cite{Gan18}, Gilbert et al. \cite{Gil18, Gil18b,Gil20,Gil21}, Graham et al. \cite{Gra11,Gra15,Gra18},
Herrmann and Schwab \cite{Her18-2,Her18-3,Her18}, Kazashi \cite{Kaz18}, Kuo and Nuyens \cite{Kuo16,Kuo18}, Kuo et al. \cite{Kuo12,Kuo15,Kuo17}, Lemieux \cite{Lem09},  Leobacher and Pillichshammer \cite{Leo14},  Nguyen and Nuyens \cite{Ngu18,Ngu19}, Nichols and Kuo \cite{Nic14} to mention just a few.

Let $D\subset \RR^d$, $d=1,2,3$, be a bounded Lipschitz domain. In this paper we consider a family of real parametric eigenvalue problems (EVP) of the form
\be \label{EVP}
-\div\big(a(\by)(\bx)\nabla \omega(\by)(\bx)\big)  
=
\lambda (\by)  \omega(\by)(\bx)
,
\,
\qquad \bx \in D
,
\ee 
with the homogeneous Dirichlet boundary condition, i.e., $\omega(\by)(\bx)|_{\partial D}=0$. We assume that the coefficient $a(\by)(\bx)$   has  an expansion of the form
\be \label{eq:rep}
a(\by)(\bx)=a_0(\bx)+\sum_{j\in \NN}y_ja_j(\bx) ,
\ee 
where $a_0$ and $(a_j)_{j\in \NN}$ belong to $L_\infty(D)$ and elements $y_j$ of the parameters $\by=(y_j)_{j\in \NN}\in U^\infty $, $U:=[-\frac{1}{2},\frac{1}{2}]$, are  independent and identically uniformly distributed on $U$. Hence, the distribution of $\by$ is given by the product measure $\rd \by =\bigotimes_{j\in \NN}\rd y_j$ on  $U^\infty$.  

We denote by $\langle \cdot,\cdot \rangle $ the inner product in $L_2(D)$ and by $V:=H_0^1(D)$ the Sobolev space of  real-valued functions with vanishing boundary in the sense of trace. The norm of the function $v\in V$ is defined by
\beqq
\|v\|_V:=\|\nabla u\|_{L_2(D)}\,.
\eeqq
The dual space $H^{-1}(D)$ of $V$ is denoted by $V^*$ and $\langle \cdot,\cdot\rangle_{V\times V^*}$ is the duality pairing on $V$ and $V^*$.
For $\by\in U^\infty$ we define the symmetric bilinear form $\Bb(\by,\cdot,\cdot):V\times  V\to \RR$ by  
\begin{equation}\label{eq-bilinearf-form}
\begin{split} 
	\Bb(\by,u,v)&:=\int_D a(\by)(\bx) \nabla u(\bx)  \cdot \nabla v(\bx)\rd \bx \,.
\end{split}
\end{equation}
The variational formulation of the parametric EVP \eqref{EVP} reads as follows. For any given $\by \in U^\infty$, find  $\big(\lambda(\by),\omega(\by)\big)\in \RR \times V$, with $\omega(\by)\not =0$   such that
\be\label{real}
\begin{split}
	\Bb(\by,\omega(\by),v)  &= \lambda(\by)\langle \omega(\by),v\rangle ,\qquad \text{for all }v\in V
	\\
	\|\omega(\by)\|_{L_2(D)}&=1\,.
\end{split}
\ee
Under some assumptions on the systems $(a_j)_{j \in \NN}$ it will be proven (see  \RefSec{sec:wellposed}) that for any $\by \in U^\infty$  the problem~\eqref{real} has a countably infinite number of eigenvalues denoted by $(\lambda_k(\by))_{k \in \NN}$. Moreover, they can be ordered in a non-decreasing sequence (associated with $\by$). 
In the present paper, we are interested in studying the analyticity of the leftmost eigenvalue $\lambda_1(\by) $ and analyzing the convergence rate of approximating the expectation
\begin{equation}\label{eq:objective}
	\E_{\by}[\lambda_1] := \int_{U^\infty} \lambda_1(\by) \rd \by
\end{equation}
by randomized QMC rules.

The eigenvalue 
problems of parametric or 
stochastic elliptic differential 
operators have been of interest for the past fifty years, see \cite{Wa66,SA72,GG07,Wi10,And12,Wi13,HKL15,ES19,HL19,Gil18,Gil18b,Gil20,Gil21} and references therein. 
These problems appear in many areas of engineering and physics, for example, in nuclear reactor physics; photonics; quantum physics; acoustic; or in electromagnetic.  In applications, the leftmost eigenvalue $\lambda_1$ and its corresponding eigenfunction $\omega_1$ usually have an important physical meaning, see \cite{Wa66,Dud76Nuc,Do99, GG12}. For example, in the model of nuclear reactor the eigenvalue $\lambda_1$ characterizes the physical state of the core reactor (critical, supercritical, or subcritical) while the eigenfunction $\omega_1$ models the associated neutron flux. Therefore, the calculation of the smallest eigenvalue $\lambda_1$ and its eigenfunction $\omega_1$ is one of the primary objectives of nuclear reactor analysis.

The analytic dependence of the eigenvalue $\lambda_1(\by)$ on the parameters $\by$  of the parametric EVP \eqref{real} has been considered in \cite{And12}. The analysis of the QMC method for $\lambda_1(\by)$  was studied in \cite{Gil18,Gil20,Gil21}. However, in these mentioned papers, the authors have not taken into account the amount of overlap between the
supports of the functions $(a_j)_{j\in \NN}$. It has been observed in many situations that, for example  in elliptic PDEs with stochastic diffusion coefficients \cite{BCDM,BCM}, disjoint supports  or finite overlap of the functions  $(a_j)_{j\in \NN}$ (such as wavelet-type representations) generally leads to simpler
analysis and a lower computational cost compared to global supports representations. 
Motivated by this fact, the present paper aims at proving the   holomorphic extensions of $\lambda_1(\by)$ to a complex domain in $\CC^\infty$ and extending the QMC convergence theory
of  \cite{Gil18} accounting
for possible locality of the supports of the functions $(a_j)_{j\in \NN}$ in the representation \eqref{eq:rep}. For the analysis relevant, it  requires  that the spectral gap $\lambda_2(\by)-\lambda_1(\by)$ is bounded away from zero uniformly in $U^\infty$. In this paper we  give a simple proof for the uniform positivity of this spectral gap under a weaker assumption compared to \cite{Gil18,Gil18b}. To do this we point out that the set $\mathcal{K}=\{ a(\by)(\bx): \ \by \in U^\infty\}$ is compact in $L_\infty(D)$  and the mapping 
$\mathcal{K}\ni a \longrightarrow \lambda_2(a) - \lambda_1(a) \in \RR \,
$
is Lipschitz continuous.

The main tool in analyzing the error of approximating the integral \eqref{eq:objective} by QMC formula  is the bound on the mixed partial derivatives $|\partial^\bnu \lambda_1 (\by) |$ with respect to parameters $\by$, where $\bnu=(\nu_j)_{j\in \NN}$ is a multi-index with finitely many non-zero entries. One of such a bound was obtained in \cite{Gil18} where the authors  have proved that
\be \label{eq:alexbound}
|\partial^{\bnu}\lambda_1(\by)| \leq  C (|\bnu|!)^{1+\varepsilon} \big( C_\varepsilon \bbeta \big)^\bnu\,,
\ee 
 for  $\bbeta=(\|a_j\|_{L_\infty(D)})_{j\in \NN}$ and  $\varepsilon$ arbitrarily close to zero. This estimate  highly depends on $\varepsilon$, in particular the constant $C_\varepsilon$ tends to infinity when $\varepsilon$ approaches zero, see \cite[Lemma 3.3.]{Gil18}. The assumptions for the QMC convergence  in \cite{Gil18}
 relied on the $p$-summability of the sequence $\bbeta$, in detail it was assumed that $\sum_{j\in \NN}\|a_j\|_{L_\infty(D)}^p <\infty$ for some $p\in (0,1)$. 
One disadvantage of the estimate \eqref{eq:alexbound} is that it does not allow the authors in \cite{Gil18} to study convergence of the QMC quadrature in the case $p=1$.

In this paper, by using Cauchy's formula for analytic functions, we give a new bound for  mixed partial derivatives of $\lambda_1(\by)$ with respect to $\by$. Under the assumption that $\brho=(\rho_j)_{j\in \NN}$ is a  sequence of positive numbers satisfying
\beqq
 \Bigg\|\sum_{j\in \NN}\rho_j|a_j|\Bigg\|_{L_\infty(D)} <\infty\,
\eeqq
we prove that
\be \label{eq:ourbound}
|\partial^{\bnu}\lambda_1(\by)| 
\leq 
K 
\frac{\bnu! } {(\eta\brho)^\bnu},
\ee
for some positive constants $K$ and $\eta$. Hence our analysis in this paper improves the result in \cite{Gil18} to the case of the sequence of functions $(a_j)_{j \in \NN}$ having disjoint supports or finite overlap. Moreover the estimate \eqref{eq:ourbound} also allows us to consider the QMC method in  the case $p=1$ which was left open in \cite{Gil18}, see Remark \ref{rem:global} and Theorem \ref{thm-main3}.

The outline of this paper is as follows. In section \ref{sec:wellposed} we prove the well-posedness of the EVP \eqref{real} and recapitulate some basic properties of eigenpairs of this problem. In particular, in this section we give a simple proof showing that the spectral gap of the EVP \eqref{real} is uniformly positive in $U^\infty$.  Section \ref{sec:bound} is devoted to prove the analytic dependence on the parameters $\by$ of   $\lambda_1(\by)$ and the corresponding eigenfunction $\omega_1(\by)$  under a weaker assumption where the locality in the supports of the system $(a_j)_{j\in \NN}$ is considered. This analytic property is then employed to bound  the derivatives of the eigenvalue $\lambda_1(\by)$ and eigenfunction $\omega_1(\by)$ with respect to $\by\in U^\infty$. In section \ref{sec:error} we apply the result in Section \ref{sec:bound} to  study the convergence of the QMC method for the expectations of $\lambda_1$ and $\mathcal{G}(\omega_1)$, where $\mathcal{G}$ is an element in $V^*$. 

\noindent
{\bf Notation.} We use standard notations.  We denote $\NN_0^\infty$ the set of all
sequences $\bnu = (\nu_j)_{j\in \NN}$ with $\nu_j\in \NN_0$. Similarly, we define
$\CC^\infty$ and $U^\infty$.  Denote by $\FF$ the set of all $\bnu\in \NN_0^\infty$ such that $\supp(\bnu):=\{ j\in \NN: \nu_j \ne 0\}$ is finite. If $\bnu\in \FF$, we define
\beqq
\bnu! := \prod_{j \in \NN}\nu_j!\,,
\qquad 
|\bnu|:=\sum_{j \in \NN}\nu_j,
\qquad
\text{and}
\qquad 
\brho^\bnu := \prod_{j \in \NN}\rho_j^{\nu_j}
\eeqq
for a sequence $\brho=(\rho_j)_{j\in \NN}$ of positive numbers.  
\section{Well-posedness of the parametric eigenvalue problems}\label{sec:wellposed}
The main purpose of this section is to give a condition on which the EVP \eqref{real} is well-posed and to show the uniform positivity of the spectral gap of this problem.

Let $\chi_1$, $\chi_2, \ldots$ be eigenvalues of the negative Laplacian on $D$ with the homogeneous boundary condition. 
They are strictly positive and the eigenvalue $\chi_1$ is isolated and non-degenerate.
By the min-max principle for the variational characterization of eigenvalues of self-adjoint operators, see, e.g.,~\cite{Bab88}, we have
\begin{align}\label{chik}
\chi_k 
=  
\min_{S_k\subset V\atop \dim(S_k)=k}
\max_{0\not = u\in S_k}
\frac
{\|u\|_V}
{\|u\|_{L_2(D)}}
.
\end{align}
When $k=1$ we have the Poincar\'e inequality 
\begin{equation*}
\|v\|_{L_2(D)} \leq \chi_1^{-1/2}\|v\|_V,
\qquad 
\text{for } v\in V
.
\end{equation*}

Throughout this paper we use the following assumption.
\begin{assumption}\label{assumption} 
 \begin{enumerate}
 	\item The functions $a_0$, $(a_j)_{j\in \NN}$ belong to $\in L_\infty(D)$ and 
 	\be \label{cond:lowera0b0c} 
 	0<\alpha_{\min} \leq a_0(\bx) \leq \alpha_{\max} <\infty\,.
 	\ee
 	\item  There exists a positive sequence $\brho=(\rho_j)_{j\in \NN}$ such that $\lim_{j\to \infty}\rho_j^{-1} =0$ and  
 $ \sum_{j\in \NN}\rho_j|a_j| $ belongs to $ L_\infty(D)$. We put
 	\be \label{cond:lambda0} 
 	\Lambda_1
 	:=
  \Bigg\|\sum_{j\in \NN}\rho_j|a_j|\Bigg\|_{L_\infty(D)} 
  \qquad 
  \text{and}
  \qquad
 		\Lambda_0
 		:=
 		  \Bigg\|\sum_{j\in \NN}\frac{1}{2}|a_j|\Bigg\|_{L_\infty(D)} 
\ee
and assume that $\Lambda_0<\alpha_{\min}$.
 \end{enumerate}
\end{assumption}

We have the following result. A proof can be found in, e.g.,  \cite[Section 2]{BCM}. 
 
\begin{lemma}\label{coe-uni} 
	Let   Assumption \ref{assumption} hold. Then the bilinear form $\Bb(\by,\cdot,\cdot)$ defined in \eqref{eq-bilinearf-form} is coercive and
bounded, uniformly in $\by$, i.e,
\beqq
\begin{split}
	\Bb(\by,v,v) 
	&\geq 
	\big(\alpha_{\min} -\Lambda_0 \big) \|v\|_V^2,
	\qquad \text{for all}\ v\in V 
\end{split}
\eeqq
and
\beqq
\begin{split}
	\Bb(\by,u,v) & \leq \big(\alpha_{\max}+\Lambda_0 \big) \|u\|_V \|v\|_V,\qquad \text{for all } u,v\in V\,.
\end{split}
\eeqq
\end{lemma}

Let $\by \in U^\infty$. For any function $f\in L_2(D)$, we consider the operator 
\beqq
T(\by):\ \ L_2(D)\ni f \longrightarrow T(\by)f \in V \subset L_2(D)
\eeqq
defined by 
\begin{equation}\label{eq-T(by)}
\Bb(\by,T(\by)f,v) = \langle f,v\rangle ,\qquad \text{for all }v\in V\,.
\end{equation}
Under Assumption \ref{assumption}, it has been proved that the operator $T(\by)$ is self-adjoint, compact, and positive from $L_2(D)$ to $L_2(D)$,  see \cite[Section 1.2.2]{Hen06}. Then, there exist a real positive sequence $\mu_k(\by)$ converging to zero and a sequence of functions $\omega_k(\by)$ with $\|\omega_k(\by)\|_{L_2(D)}=1$ such that $T(\by)\omega_k(\by)=\mu_k(\by)\omega_k(\by)$. Putting $\lambda_k(\by)=\frac{1}{\mu_k(\by)}$ we  obtain 
\be \label{eq:eigen-k}
	\Bb(\by,\omega_k(\by),v)  = \lambda_k(\by)\langle \omega_k(\by),v\rangle ,\qquad \text{for all }v\in V.
\ee
The pair $\big(\lambda_k(\by),\omega_k(\by)\big)\in \RR \times V$ is called the eigenpair of the bilinear form $\Bb(\by,\cdot,\cdot)$. We have the following estimates.
\begin{lemma}\label{lem:boundLambda}
	Under  Assumption \ref{assumption}, for any $k \in \NN$ and any $\by \in U^\infty$  we have
	\begin{align*}
(\alpha_{\min}-\Lambda_0 )\chi_k
	\ \leq\
	\lambda_k(\by) \
	\leq\
	 (\alpha_{\max} +\Lambda_0 )\chi_k 
	\end{align*}
	and
	\beqq
	\|\omega_k(\by)\|_V
\	\leq \
	\bigg(
		\frac{(\alpha_{\max} +\Lambda_0 )\chi_k }{ \alpha_{\min} -\Lambda_0  }
	\bigg)^{1/2}
	.
	\eeqq
\end{lemma}
\begin{proof}
Using the min-max principle, Lemma \ref{coe-uni}, and~\eqref{chik} we obtain 
\be \label{up-bound}
\begin{split} 
\lambda_k(\by)
&
 =
 \min_{S_k\subset V\atop \dim(S_k)=k}
 \max_{0\not = u\in S_k}
 \frac{\Bb(\by,u,u)}{\langle u,u \rangle}
 \\
 & \leq  \min_{S_k\subset V\atop \dim(S_k)=k}\max_{0\not = u\in S_k}\frac{(\alpha_{\max} +\Lambda_0 )\langle \nabla u\cdot\nabla u \rangle }{ \langle u,u \rangle} 
 =  (\alpha_{\max} +\Lambda_0 )\chi_k 
 .
\end{split}
\ee 
Similarly, we have
\begin{align*} 
	\lambda_k(\by)
	&
	=\min_{S_k\subset V\atop \dim(S_k)=k}\max_{0\not = u\in S_k}\frac{\Bb(\by,u,u)}{\langle u,u \rangle}
	\\
	& \geq  \min_{S_k\subset V\atop \dim(S_k)=k}\max_{0\not = u\in S_k}\frac{(\alpha_{\min}-\Lambda_0 )\langle \nabla u,\nabla u\rangle  }{ \langle u,u \rangle}
 = (\alpha_{\min}-\Lambda_0 )\chi_k
 .
\end{align*}
Furthermore, taking $v = \omega_k(\by)$ as a test function in \eqref{eq:eigen-k}, we obtain $\Bb\big(\by,\omega_k(\by),\omega_k(\by)\big)=\lambda_k(\by)$. This and Lemma \ref{coe-uni} lead to
\beqq
\|\omega_k(\by)\|_V^2
\leq 
\frac{\lambda_k(\by)}{ \alpha_{\min} -\Lambda_0  }
 \leq 
\frac{(\alpha_{\max} +\Lambda_0 )\chi_k }{ \alpha_{\min} -\Lambda_0  }
\eeqq
which is the needed claim.
\hfill
\end{proof}

To prove the uniformly positiveness of spectral  gap $\lambda_2(\by)-\lambda_1(\by)$ in $U^\infty$ we need an auxiliary lemma. 
\begin{lemma}\label{lem:compact}
Let $\psi_j \in L_\infty(D)$ for all $j \in \NN_0$. If there exists a positive sequence $\bvarrho=(\varrho_j)_{j\in \NN}$ such that 
	$
	\sum_{j\in \NN} \varrho_j|\psi_j|  
	\in L_\infty(D)$ and $\lim_{j\to \infty}\varrho_j^{-1}=0$, then
	the set of functions
\beqq
\psi(U^\infty)
:=
\bigg\{
 \psi(\by):= \psi_0 +\sum_{j\in \NN}y_j\psi_j :\ 
  \by=(y_j)_{j\in \NN}\in U^\infty\bigg\}
\eeqq is compact in $L_\infty(D)$.
\end{lemma}
\begin{proof} One can follow the argument in \cite[Lemma 2.7]{CoDe} by showing that for every sequence in $\psi(U^\infty)$ we can extract a subsequence whose limit belongs to $\psi(U^\infty)$. In the following we show that $\psi(U^\infty)$ can be approximated by a subspace of finite dimension in $L_\infty(D)$. Indeed for $n\in \NN$ and $\psi(\by)\in \psi(U^\infty)$ we approximate $\psi(\by)$ by $\psi_0+\sum_{j=1}^{n-1}y_j\psi_j$. Since	$
	\sum_{j\in \NN} \varrho_j|\psi_j|  
	$ 
	is bounded in $D$  we have
\begin{align*}
	\Bigg\|\psi(\by)-\psi_0-\sum_{j=1}^{n-1}y_j\psi_j\Bigg\|_{L_\infty(D)}
	& =\Bigg\|  \sum_{j=n}^{\infty}y_j\psi_j\Bigg\|_{L_\infty(D)}
	\\
	&\leq 	\frac{1}{2}\sup_{j\geq n}\varrho_j^{-1}  \Bigg\| \sum_{j=n}^{\infty}\varrho_j|\psi_j|\Bigg\|_{L_\infty(D)}\leq C \sup_{j\geq n}\varrho_j^{-1}.
\end{align*}
Due to 	$\lim_{j\to \infty}\varrho_j^{-1}=0$ we have $\lim_{n\to \infty}\big(\sup_{j\geq n}\varrho_j^{-1}\big)=0$ which implies that the set $\psi(U^\infty)$ is compact in $L_\infty(D)$.
\hfill
\end{proof}

 We define  
 $$\mathcal{K}
:=
\big\{
a(\by) \in L_\infty(D) ,\ \by\in U^\infty
\big\}.$$
In the following proposition we will show that  the map
\begin{align*}
\lambda_k:\ \mathcal{K}\ni a \longrightarrow \lambda_k(a) \in \RR \,
\end{align*}
 is Lipschitz continuous with respect to $a$. As a consequence we conclude  that the spectral gap of the EVP \eqref{real} is  uniformly positive in $U^\infty$.
\begin{proposition}\label{prop:properties} 
		Under Assumption \ref{assumption}, for any $\by\in U^\infty$ the EVP \eqref{real} has the following properties:
	\begin{enumerate}
		\item There are countably-many eigenvalues $(\lambda_k(\by))_{k\in \NN}$ which are all positive, have finite multiplicity and accumulate at infinity. Counting multiplicities we can write
		\beqq
		0<\lambda_1(\by) < \lambda_2(\by) \leq \ldots
		.		
		\eeqq
		Additionally, $\lambda_1(\by)$ is isolated and non-degenerate.
		\item There exist four positive constants $\gamma_{\min}$, $\gamma_{\max}$, $\delta_{\min}$ and $\delta_{\max}$ independent of $\by$ such that
		\be \label{gamma}
		\begin{split}
  0< \gamma_{\min}\leq \gamma(\by):= \lambda_2(\by)-\lambda_1(\by) \leq \gamma_{\max} < \infty
  ,
\end{split}
\ee
and 
\be \label{delta}
\begin{split} 
0<
\delta_{\min}\leq \delta(\by):=\frac{\lambda_2(\by)-\lambda_1(\by)}{\lambda_1(\by)}\leq \delta_{\max}
< \infty
.
		\end{split}
\ee 
		\item For any $k \in \NN$ the eigenvalue $\lambda_k(\by)$ is Lipschitz	continuous in $\by$ (with $ \ell_\infty(\NN)$-norm).
	\end{enumerate}  
\end{proposition}
\begin{proof} 
	Because the bilinear form $\Bb(\by,\cdot,\cdot)$ is coercive and
	bounded the first claim has been proved in~\cite[Section 2]{And12}, see also~\cite[Section 2.1]{Gil18} (by using the Krein--Rutman theorem). 
	
Since $T(\by)$ are self-ajoint, compact, and positive  operators, from \cite[Theorem 2.3.1]{Hen06}, for  $\by, \tilde{\by}\in U^\infty$ we have
\beqq
|\mu_k(\by)-\mu_k(\tilde{\by})| \leq \|T(\by)-T(\tilde{\by})\|_{L_2(D)\to L_2(D)}
\eeqq	
which is equivalent to
\be \label{1-2}
|\lambda_k(\by) - \lambda_k(\tilde{\by})| \leq \lambda_k(\by)\lambda_k(\tilde{\by})  \|T(\by)-T(\tilde{\by})\|_{L_2(D)\to L_2(D)}\,.
\ee 
In the following we will show that 
\begin{equation}\label{eq-T(y)-T(y)}
\|T(\by)-T(\tilde{\by})\|_{L_2(D)\to L_2(D)} \leq \frac{1}{\sqrt{\chi_1}} \frac{\alpha_{\max}+\Lambda_0}{(\alpha_{\min}-\Lambda_0)^2} \|a(\by)-a(\tilde{\by})\|_{L_\infty(D)} 
\end{equation}
by
 using the same argument as in the proof of  \cite[Proposition 2.3]{Gil18}. From \eqref{eq-T(by)} we have
\begin{align*}
	\Bb(\by,T(\by)f,v) = \Bb(\tilde{\by},T(\tilde{\by})f,v) 
\end{align*}
for $f\in L_2(D)$ and $v\in V$. This implies
\begin{align*}
	\Bb(\by,T(\by)f-T(\tilde{\by})f,v) & = 	\Bb(\tilde{\by}, T(\tilde{\by})f,v)-	\Bb(\by,T(\tilde{\by})f,v)
	\\
	&= \int_D \big[a(\by)(\bx)-a(\tilde{\by})(\bx)\big] \nabla T(\tilde{\by})f(\bx) \cdot  \nabla v(\bx)\rd \bx.
\end{align*}
With $v=T(\by)f-T(\tilde{\by})f$, by Lemma \ref{coe-uni} and the Cauchy–Schwarz inequality
we get
$$
(\alpha_{\min} -\Lambda_0)  \|T(\by)f-T(\tilde{\by})f\|_V^2 \leq \|a(\by)-a(\tilde{\by})\|_{L_\infty(D)} \|T(\tilde{\by})f\|_V \|T(\by)f-T(\tilde{\by})f\|_V.
$$
Using Lax–Milgram Theorem and Poincar\'e inequalities  we find
$$
\|T(\tilde{\by})f\|_V\leq \frac{\alpha_{\max}+\Lambda_0}{\alpha_{\min}-\Lambda_0}\|f\|_{V^*} \leq \frac{1}{\sqrt{\chi_1}}\frac{\alpha_{\max}+\Lambda_0}{\alpha_{\min}-\Lambda_0}\|f\|_{L_2(D)}.
$$
Hence
$$
 \|T(\by)f-T(\tilde{\by})f\|_V \leq \frac{1}{\alpha_{\min} -\Lambda_0 } \|a(\by)-a(\tilde{\by})\|_{L_\infty(D)}\cdot  \frac{1}{\sqrt{\chi_1}}\frac{\alpha_{\max}+\Lambda_0}{\alpha_{\min}-\Lambda_0}\|f\|_{L_2(D)}
$$
which implies \eqref{eq-T(y)-T(y)}. 
Now by Lemma \ref{lem:boundLambda} we infer the existence of a constant $C_k>0$ such that
\begin{equation}  \label{eq-continuos}
|\lambda_k(\by) - \lambda_k(\tilde{\by})| \leq C_k \|a(\by)-a(\tilde{\by})\|_{L_\infty(D)}\,.
\end{equation}
Consequently, we obtain
\beqq
\begin{split} 
\big|\big[\lambda_2(\by)-\lambda_1(\by)\big] & - \big[\lambda_2(\tilde{\by})-\lambda_1(\tilde{\by})\big] \big|
\leq 
(C_1+C_2) \|a(\by)-a(\tilde{\by})\|_{L_\infty(D)}\,
\end{split}
\eeqq
or with $a=a(\by)$ and $\tilde{a}=a(\tilde{\by})$ we can write 
\beqq
\big|\big[\lambda_2(a)-\lambda_1(a)\big]  - \big[\lambda_2(\tilde{a})-\lambda_1(\tilde{a})\big] \big|
\leq 
(C_1+C_2) \|a-\tilde{a}\|_{L_\infty(D)}\,.
\eeqq
This implies that the map $a \mapsto \lambda_2(a) -\lambda_1(a)$ is continuous in the set $\mathcal{K}\subset L_\infty(D)$. We know from~\RefLem{lem:compact} that  $\mathcal{K}$ is compact in $L_\infty(D)$. Moreover, since $0<\lambda_2(\by)-\lambda_1(\by)<\infty$ for all $\by\in U^\infty$, we conclude  that there exit two constants $\gamma_{\min}$ and $\gamma_{\max}$ such that 
$$0
<
\gamma_{\min}\leq \gamma(\by)=
 \lambda_2(\by)-\lambda_1(\by) 
 \leq 
 \gamma_{\max}<\infty
 .
 $$
 The estimate \eqref{delta}  then follows from~\RefLem{lem:boundLambda}. This is the second claim. 

Finally, we have
\beqq
\begin{split}
\|a(\by)-a(\tilde{\by})\|_{L_\infty(D)} 
&
= \Bigg\| \sum_{j\in \NN}(y_j-\tilde{y}_j) a_j\Bigg\|_{L_\infty(D)}
 \leq 2\|\by-\tilde{\by}\|_{\ell_\infty(\NN)}\Bigg\|\sum_{j\in \NN}\frac{1}{2}|a_j| \Bigg\|_{L_\infty(D)}
 \\
 &
 = 2\Lambda_0  \|\by-\tilde{\by}\|_{\ell_\infty(\NN)}\,.
\end{split}
\eeqq
Inserting this into \eqref{eq-continuos} we obtain the last claim. The proof is finished.
\hfill  
\end{proof}
\section{Parametric analyticity and bound of mixed derivatives}\label{sec:bound}
The analytic dependence of the eigenpair $(\lambda_1,\omega_1)$ on the parameters $\by$  of the parametric EVP \eqref{real} has been studied in \cite{And12} under the assumption that $\sum_{j\in \NN}\|a_j\|_{L_\infty(D)}^p <\infty$, $p\in (0,1]$.
In this section, we will extend this result to a weaker assumption where the locality in the supports of the system $(a_j)_{j\in \NN}$ is considered. Afterward, we use Cauchy's formula to estimate  the derivatives of the eigenvalues $\lambda_1(\by)$ and eigenfunctions $\omega_1(\by)$ with respect to $\by\in U^\infty$. We assume in this section that  $L_2(D)$ and $V$ are complex-valued function spaces. We consider the coefficients of the form
\beqq
a(\bz)(\bx)=a_0(\bx)+\sum_{j\in \NN}z_ja(\bx),
\eeqq
where $\bz=(z_j)_{j\in \NN}\in \CC^\infty$. We define the associated sesquilinear forms $\Bb(\bz,\cdot,\cdot)$ and $ \Bb_j(\cdot,\cdot)$ from $ V\times V$ to  $\CC$ by
$$ 
\Bb(\bz,u,v)
:= 
\int_D a(\bz)(\bx) \nabla u(\bx)  \cdot \overline{ \nabla v(\bx)}\,\rd \bx 
$$
and
$$
	\Bb_j(u,v)
	:=
	\int_D a_j(\bx) \nabla u(\bx)  \cdot \overline{ \nabla v(\bx)}\rd \bx\,.
$$

Let $\Ll(V,V^*)$ denote the set of all continuous linear mappings from $V$ to $V^*$.
We define  $A(\bz)$ and $A_j \in \Ll(V,V^*)$ for $j \in \NN_0$ the operators corresponding to  $\Bb(\bz,\cdot,\cdot)$ and $\Bb_j(\cdot,\cdot)$ by  identifications
$$
	\Bb(\bz,u,v)=\langle u,A(\bz)v\rangle_{V\times V^*},
	\qquad 
	\Bb_j(u,v) =\langle u,A_jv\rangle_{V\times V^*} ,\qquad\text{for }u,v\in V\,.
$$   
The complex version of the EVP~\eqref{real} reads as follows. Find $\big(\lambda(\bz),\omega(\bz)\big)\in \CC \times V$, with $\omega(\bz)\not =0$ such that
\begin{equation} \label{complex}
\begin{aligned}
	\Bb(\bz,\omega(\bz),v) &= \lambda(\bz) \langle \omega(\bz),v\rangle ,\qquad \text{for all }v\in V
	\\
	\|\omega(\bz)\|_{L_2(D)}&=1\,.
\end{aligned}
\end{equation}
This problem is well-posed as long as there exist positive constants $C$ and $\gamma$ (might depend on $\bz$) such that
\begin{align*}
 &|\Bb(\bz,u,v)|  \leq C \|u\|_V \|v\|_V,\qquad \text{for all }u,v\in  V
\\
&\inf_{0\not = u\in V}\sup_{0\not = v\in V}\frac{|\Bb(\bz,u,v)|}{\|u\|_V \|v\|_V} \geq \gamma,
\\
\text{and } &\sup_{u\in V}|\Bb(\bz,u,v)| >0,\qquad\text{for all } 0 \not = v\in V\,,
\end{align*}
see~\cite[Section 2]{And12}. 

Before formulating our main  results in this section we recall the notion of separate holomorphy of countable product spaces over $\CC$. Let $(Z_j)_{j\in \NN}$ be a family of Banach spaces over $\CC$ and $Y$ also a Banach space over $\CC$. Let $S\subset \bigtimes_{j\in \NN}Z_j$ be an open set and $\bz=(z_j)_{j\in \NN}\in S$. For a finite set $J\subset \NN$ we denote
$$
S_J(\bz):= \big\{ (y_j)_{j\in J}: \ \exists (v_j)_{j\in \NN}\in S \ \text{with}\ v_j=y_j, \ j\in J\  \text{and}\ v_j=z_j, \ j\not \in J \big\}\,.
$$
We say that the map $u: S\to Y$ is separately holomorphic if for every finite set $J\subset \NN$ and $\bz\in S$ the map 
$u$ is holomorphic as a function of variables in $S_J(\bz)$.
We have the following result. 
\begin{theorem}\label{thm:analytic} 
	Let Assumption \ref{assumption} hold and  $\gamma_{\max}$, $\delta_{\min}$ be given in \RefProp{prop:properties}. For $\varepsilon\in (0,1)$ we put $\kappa :=\frac{1-\varepsilon}{2(1+\delta_{\min}^{-1})}<1$ and define the sequence $\btau=(\tau_j)_{j\in \NN}$ where
\begin{equation} \label{eta}
\tau_j:=\eta_\varepsilon \rho_j\quad
\text{with }\quad
 \eta_\varepsilon:=(1-\varepsilon)\frac{\alpha_{\min} -\Lambda_0  }{2\Lambda_1 (1+\delta_{\min}^{-1})} 
\end{equation}
and
\begin{equation*}
\mathcal{E}(\btau):=\bigtimes_{j\in \NN} \mathcal{E}_j(\btau)\quad
\text{with }\quad
 \mathcal{E}_j(\btau):=\bigg\{
z_j\in \CC: \dist\Big(z_j,\Big[-\frac{1}{2},\frac{1}{2}\Big]\Big)<\tau_j\bigg\}.
\end{equation*}

Then the eigenpair $(\lambda_1, \omega_1)$ of the EVP \eqref{real} can be extended to  separate holomorphic functions on $\Ee(\btau)$. Moreover, we have 
	\be \label{up-lambda}
	\sup_{\bz\in \Ee(\btau)}|\lambda_1(\bz)| \leq \frac{\gamma_{\max}}{2} +   \big(\alpha_{\max} +\Lambda_0 )\chi_1  =: K_{\lambda}
	,
	\ee 
and
\be \label{up-omega}
	\sup_{\bz\in \Ee(\btau)}\|\omega_1(\bz)\|_V 
	\leq 
	\bigg(\frac{  \gamma_{\max}  + 2(\alpha_{\max} +\Lambda_0)\chi_1}{2(1-\kappa)(\alpha_{\min} -\Lambda_0) }\bigg)^{1/2}
	=:K_{\omega}
	.
\ee 
\end{theorem}
\begin{proof}We follow the proof of \cite[Theorem 2.13]{And12}. For $\bz=(z_j)_{j\in \NN}\in \Ee(\btau) $ we take $\by=(y_j)_{j\in \NN}\in U^\infty$ such that $|z_j-y_j|<\tau_j$ and denote  $\bzeta:=\bz-\by=(\zeta_j)_{j\in \NN}$. Hence, $A(\bz)$ can be written as
$
A(\bz) = A(\by) +  B(\bzeta),
$
where $B(\zeta):=\sum_{j\in \NN}\zeta_jA_j$. 
Define the complex-analytic  operator-valued function 
$$
t \to A(\by) +  t B(\bzeta)
.
$$ 

From~\RefLem{coe-uni} we know that
\beqq
\begin{split}
\langle v,A(\by)v\rangle_{V\times V^*}&\geq \big(\alpha_{\min} -\Lambda_0 \big) \|v\|_V^2,\qquad \text{for all}\ v\in V 
\end{split}
.
\eeqq
We now estimate the norm $\|B(\bzeta)\|_{\Ll(V,V^*)}$. Since $|\zeta_j|< \tau_j$ for all $j\in \NN$, under Assumption \ref{assumption} we obtain
\begin{equation*}
\begin{split}
\langle u,B(\bzeta)v\rangle_{V\times V^*}
&
\leq 
 \int_D
  \Bigg| \sum_{j\in \NN}\tau_j a_j(\bx)\Bigg| \big| \nabla u(\bx)  \cdot \overline{ \nabla v(\bx)}\big|\,\rd \bx  
\leq 
\eta_\varepsilon \Lambda_1   \int_D |\nabla u(\bx)   \cdot \overline{ \nabla v(\bx)}|\,\rd \bx
\\
&
\leq \frac{(1-\varepsilon)\big(\alpha_{\min} -\Lambda_0 \big) }{2 (1+\delta_{\min}^{-1})} \|u\|_V \|v\|_V
 =\kappa \big(\alpha_{\min} -\Lambda_0 \big)  \|u\|_V \|v\|_V\,.
\end{split}
\end{equation*}
This implies $$\|B(\bzeta)\|_{\Ll(V,V^*)}\leq \kappa \big(\alpha_{\min} -\Lambda_0 \big) .$$
Thus, according to~\cite[Theorem 2.6 and Corollary 2.8]{And12} we can analytically extend $\by \to \lambda_1(\by)$ to a complex-valued function $t \to \tilde{\lambda}(\by+t\bzeta)$  in the disk
\beqq
\bigg\{t\in \CC, \ |t|< \frac{1}{1-\varepsilon} \bigg\} = \bigg\{t\in \CC, \ |t|< \frac{1}{2\kappa(1+\delta_{\min}^{-1})} \bigg\} 
.
\eeqq
Moreover, we have $\tilde{\lambda}(\by+t\bzeta)$ is an isolated and non-degenerated eigenvalue of $A(\by) +  t B(\bzeta)$.

 Thus $\tilde{\lambda}(\by+\bzeta)$  is a candidate for the holomorphic
 extension $\lambda_1(\bz)$ of the parametric eigenvalue.
It has been shown in  \cite[Theorem 2.13]{And12} that $ \tilde{\lambda}(\by+\bzeta)$ is independent of the choice $\by\in U^\infty$ and $\bzeta$ satisfying $\bz=\by+\bzeta$. 
Therefore, $\lambda_1(\bz):=\tilde{\lambda}(\by+\bzeta)$ is well-defined. Using the same argument  at the end of the proof of \cite[Theorem 2.13]{And12} we obtain the holomorphic extension of $\lambda_1$  on $\Ee(\btau)$.

 Similar considerations apply for eigenfunction $\omega_1$. Checking the proofs of \cite[Theorem 2.6]{And12} and \cite[Theorems XII.8 and XII.11]{Ree78} we find that $\lambda_1(\bz)$ satisfies
\beqq
|\lambda_1(\bz)-\lambda_1(\by)|\leq \frac{1}{2}\gamma(\by)\leq \frac{\gamma_{\max}}{2},
\eeqq
where $\gamma(\by)$ is given in \eqref{gamma}. Consequently, we obtain from \eqref{up-bound}
\beqq
|\lambda_1(\bz)| \leq \frac{\gamma_{\max}}{2} + \lambda_1(\by) \leq  \frac{\gamma_{\max}}{2} + (  \alpha_{\max} +\Lambda_0 )\chi_1 \,.
\eeqq
To show uniformly boundedness of  $\|\omega_1(\bz)\|_V$ in $\Ee(\btau)$ we use $v=\omega_1(\bz)$ as a test function in \eqref{complex} to get
$$
|\Bb(\bz,\omega_1(\bz),\omega_1(\bz))|=|\lambda_1(\bz)|\,.
$$
This together with 
$$
\Bb\big(\bz,\omega(\bz),\omega(\bz)\big)\geq (\alpha_{\min} -\Lambda_0 )(1-\kappa)\|\omega_1(\bz)\|_V^2
$$  
leads to
\beqq
\|\omega_1(\bz)\|_V^2 \leq \frac{|\lambda_1(\bz)|}{(1-\kappa)(\alpha_{\min} -\Lambda_0 ) } \,.
\eeqq
The proof is completed.
\hfill
\end{proof}

The analyticity of the eigenpair $(\lambda_1(\by),\omega_1(\by))$ leads to the following.
\begin{theorem}\label{Thm:derivatives} 
		Let Assumption \ref{assumption} hold. Then for any $\by\in U^\infty$ and any $\bnu\in \FF$ the partial mixed derivative of eigenvalue $\lambda_1(\by)$ and eigenfunction $\omega_1(\by)$ of the EVP \eqref{real} can be estimated  á
\begin{equation*}
	|\partial^{\bnu}\lambda_1(\by)| 
	\leq 
	K_{\lambda}
	\frac{\bnu! } {(\eta\brho)^\bnu}
\,	,
\qquad \text{and}\qquad
\|\partial^{\bnu}\omega_1(\by)\|_V \leq K_\omega
\frac{\bnu! } {(\eta \brho)^\bnu}
,
\end{equation*}
where $ \eta:=\frac{\alpha_{\min} -\Lambda_0  }{2\Lambda_1 (1+\delta_{\min}^{-1})} $ and $K_{\lambda}$ and $K_\omega$ are given in  \eqref{up-lambda} and \eqref{up-omega}.
\end{theorem}
\begin{proof} 
Let $\varepsilon\in (0,1)$ and $\eta_\varepsilon$ bee given in \eqref{eta}.	 From \RefThm{thm:analytic} we know that eigenpair $(\lambda_1, \omega_1)$ of the EVP \eqref{real} can be extended to   separately complex-analytic functions on $\Ee(\tilde{\btau})$ with $\tilde{\btau}=\eta_{\varepsilon/2}\brho$. Hence, for any $\by\in U^\infty$ and  $\bnu \in \FF$ with $\mfu:=\supp(\bnu)$ applying Cauchy's formula gives 
	\begin{equation*}
	\partial^{\bnu}\lambda_1(\by) 
	=
	\frac{\bnu!}{(2\pi i)^{|\mfu|}}
	\int_{\mathcal{C}_\mfu(\by,\btau)} 
	\frac{\lambda_1(\bz_\mfu)}
	{\prod_{j\in \mfu}  (z_j-y_j)^{\nu_j+1}}
\prod_{j\in \mfu}	\rd z_j  
	,
\end{equation*}
	where 
\begin{equation*}
	\mathcal{C}_\mfu(\by,\btau)
	:= 
	\bigtimes_{j\in \mfu} 
	\mathcal{C}_{j}(\by,\btau)
	\qquad
	\text{with }\quad
	\mathcal{C}_j (\by,\btau):= \big\{ z_j \in \CC: |z_j-y_j|=\tau_j\big\}
	,
\end{equation*}
	and $$\bz_\mfu\in 	\mathcal{C}^*_\mfu(\by,\btau):= \big\{(z_j)_{j\in \NN}\in \CC^\infty: z_j\in \mathcal{C}_{j}(\by,\btau)\ \text{if}\  j\in \mfu\ \text{and}\ z_j=y_j\ \text{if}\ j\not \in \mfu\big\}.$$
Using \eqref{up-lambda} we obtain
\begin{equation*}
	\begin{aligned} 
|\partial^{\bnu}\lambda_1(\by)| 
& 
\leq 
\frac{\bnu!}{(2\pi )^{|\mfu|}}
\sup_{\bz_\mfu \in \mathcal{C}^*_\mfu(\by,\btau)}|\lambda_1(\bz_\mfu)|
\int_{\mathcal{C}_\mfu(\by,\btau)} 
 \frac{ \prod_{j\in \mfu}	\rd z_j  }{\prod_{j\in \mfu}  |z_j-y_j|^{\nu_j+1}}
 \\
& \leq
 K_{\lambda}
\frac{\bnu! } {\btau^\bnu}  = K_{\lambda}
\frac{\bnu! } {(\eta_\varepsilon \brho)^\bnu}  \,.
	\end{aligned}
\end{equation*}
 Similar considerations give estimate for $\|\partial^{\bnu}\omega_1(\by)\|_V$. Since these bounds hold for any $\varepsilon\in (0,1)$ we obtain the desired results. The proof is completed.
	\hfill \end{proof}
\begin{remark}\label{rem:global}
We give a comment when the system $(a_j)_{j\in \NN}$ has arbitrary supports. Let $\beta_j=\|a_j\|_{L_\infty(D)}$ and $\bbeta=(\beta_j)_{j\in \NN}$. Assume that the second condition in $\eqref{cond:lambda0}$ is replaced by
$
		\Lambda_0
:=
\sum_{j\in \NN}\frac{1}{2}\beta_j < \alpha_{\min}\,.
$ This assumption guarantees that the set
 $\mathcal{K}
:=
\big\{
a(\by) \in L_\infty(D) ,\ \by\in U^\infty
\big\}$ is compact in $L_\infty(D)$, see \cite[Lemma 2.7]{CoDe}. As a consequence, Proposition \ref{prop:properties} holds. 
Now for each $\bnu\in \FF$ fixed we define the sequence $\brho_\bnu=(\rho_j)_{j\in \NN}$ with $\rho_j=\frac{\nu_j}{|\bnu| \beta_j}$ if $j\in \supp(\bnu)$ and $\rho_j=0$ otherwise. From this we have 
$ 	\Lambda_1
:=
\big\|\sum_{j\in \NN}\rho_j|a_j|\big\|_{L_\infty(D)} =1$.  Next, following argument  in the proof of Theorems \ref{thm:analytic} and \ref{Thm:derivatives} we can show that 
	\begin{equation*} 
|\partial^{\bnu}\lambda_1(\by)| 
\leq 
K_{\lambda}
\frac{\bnu! } {(\eta\brho_\bnu)^\bnu} = K_{\lambda}
\frac{\bnu! |\bnu|^{|\bnu|} } {\bnu^\bnu} \Big(\frac{\bbeta}{\eta} \Big)^\bnu,
\end{equation*}
with $ \eta:=\frac{\alpha_{\min} -\Lambda_0  }{2  (1+\delta_{\min}^{-1})} $. Employing the estimate $\frac{|\bnu|^{|\bnu|}}{\bnu^\bnu} \leq e^{|\bnu|} \frac{|\bnu|!}{\bnu!}$, see, e.g., \cite[Page 61]{CoDe} we get 
\beqq
|\partial^{\bnu}\lambda_1(\by)| 
\leq K_\lambda |\bnu|! \Big( \frac{e\bbeta}{\eta}\Big)^\bnu \,.
\eeqq
A similar argument applies for $\omega_1(\by)$. This estimate improves \eqref{eq:alexbound} and can be used to consider the QMC error in the case $(\beta_j)_{j\in \NN}\in \ell_1(\NN)$ which was excluded in \cite{Gil18}.  
\end{remark}
 
\section{Analysis of the QMC method}\label{sec:error}
In this section we apply the estimate of derivatives of $\lambda_1(\by)$ and $\omega_1(\by)$ with respect to the parameters $\by$ to analysize the convergence rate of QMC method for $\E_\by(\lambda_1)$ and $\E_\by(\mathcal{G}(\omega_1))$ where $\mathcal{G}\in V^*$.

We review some basic results about the QMC quadratures for approximating the $s$-dimension integrals, following \cite{Kuo16}.
For a measurable function $F: U^s\to \RR$ we seek to approximate the integral of the form
\beqq
I_s(F): =\int_{U^s}F(\bxi)\rd\bxi.
\eeqq
To approximate $I_s(F)$ we use the randomly shifted lattice rule which is given by the QMC quadrature 
\begin{equation} \label{brace}
Q_{s,N}^\varDelta(F)=\frac{1}{N}\sum_{i=1}^NF\bigg(\Big\{\frac{i\bz}{N}+\varDelta\Big\}-\frac{1}{2}\bigg),
\end{equation}
where $\bz\in \NN^s$ is the generating vector and $\varDelta$ is a random shift which is uniformly distributed over the cube $(0,1)^s$. The braces in \eqref{brace} indicate that we take the fractional parts of each component in a vector. We want to evaluate the root-mean-square error given by
\begin{equation*}
\sqrt{\mathbb{E}^\varDelta\big(|I_s(F)-Q_{s,N}^\varDelta(F)|^2\big)}\,,
\end{equation*} 
where $\mathbb{E}^\varDelta$ is the expectation with respect to the random shift $\varDelta$. 

It is well-known that good randomly shifted lattice rules can be constructed to achieve the optimal rate of convergence close to $\mathcal{O}(n^{-1})$ provided that integrand lies a certain weighted Sobolev space.
Denote $[s]=\{1,\ldots,s\}$ and let $\bgamma=(\gamma_\mfu)_{\mfu\subseteq [s]}$ be a sequence of positive weights. We define the weighted Sobolev space of mixed first order derivatives $\mathcal{W}_{\bgamma}(U^s)$ as the collection of all functions $F: U^s\to \RR$ such that
\begin{equation*}
\|F\|_{\Ww_\bgamma(U^s)}^2=\sum_{\mfu\subseteq [s]}\frac{1}{\gamma_{\mfu}} \int_{U^{|\mfu|}}\Bigg(\int_{U^{|\bar{\mfu}|}}\frac{\partial ^{|\mfu|}
	F}{\partial \bxi_{\mfu}}\big(\bxi\big) \rd \bxi_{\bar{\mfu}} \Bigg)^2\rd\bxi_{\mfu}<\infty \,.
\end{equation*}
Here  $\bar{\mfu}:=[s]\backslash \mfu$ and $\frac{\partial ^{|\mfu|}
	F}{\partial \bxi_{\mfu}}$ denotes the mixed first derivatives of $F$ with respect to the variable $\bxi_{\mfu}=(\xi_j)_{j\in \mfu}$. The weight sequence $(\gamma_{\mfu})_{\mfu\subseteq [s]}$ is associated with each subset of variables to moderate the relative importance between the different sets of variables. With an appropriate choice of weight we can get the error bound independent of the dimension $s$. Moreover, we need some structure of the weight for the Component-by-component (CBC) construction cost to be feasible. Different types of weights have been considered depending on the problem and the estimation of  $\frac{\partial ^{|\mfu|}
	F}{\partial \bxi_{\mfu}}$. In the case of product weights the cost of the fast CBC algorithm for constructing a randomly shifted lattice rule with $N$ points  is $\mathcal{O}(sN\log N)$ while $ \mathcal{O}(sN\log N+ s^2N)$ operations needed in the case of product and order dependent weights, see \cite[Section 5]{Kuo16}.

We have the following result on the error of the QMC quadrature \eqref{brace}, see, e.g., \cite[Theorem 5.1]{Kuo16}. 
\begin{proposition}\label{QMC} Let $s\in \NN$ and $(\gamma_j)_{j=1}^s$ be a positive sequence. We define the product weight by  $\bgamma=(\gamma_{\mfu})_{\mfu \subseteq [s]}$ where $\gamma_\mfu=\prod_{i\in \mfu }\gamma_j$. Then a randomly shifted lattice rule with $N$ points can be constructed in $\mathcal{O}(sN\log N)$ operations using the fast CBC algorithm such that for every $F\in \Ww_{\gamma}(U^s)$ and for every $\lambda\in (1/2,1]$ there holds the error bound
\begin{align*}
	\sqrt{\mathbb{E}^\varDelta\big(|I_s(F)-Q_{s,N}^\varDelta(F)|^2\big)}
\leq \Bigg( \sum_{\mfu \subseteq [s]}\gamma_{\mfu}^\lambda \bigg(\frac{2\zeta(2\lambda)}{(2\pi^2)^\lambda} \bigg)^{|\mfu|}\Bigg)^{\frac{1}{2\lambda}} \varphi(N)^{-\frac{1}{2\lambda}}	\|F\|_{\Ww_\bgamma(U^s)} ,
\end{align*}
where $\varphi(N)$ denotes Euler's totient function and $\zeta(x):=\sum_{k\in \NN}k^{-x}$ denotes the Riemann zeta function. 
\end{proposition}
\begin{remark}It is known that for any fixed $\delta\in (0,1)$,  we have $\frac{\varphi(N)}{N^\delta}\to \infty$ when $N\to \infty$. 
	If $N$ is a prime then we have $\varphi(N)=N-1$. We can verify that $\varphi(N)>\frac{N}{9}$ for $N\leq 10^{30}$. Hence, in practice one can replace $\varphi(N)$ by $N$ multiplying with an appropriate constant.
\end{remark}

For $\by=(y_j)_{j\in \NN}\in U^\infty$ we denote $\by_{s}=(y_1,\ldots,y_s,0,\ldots)$ and $\by_{\bar{s}} = (0,\ldots,0,y_{s+1}, y_{s+2},\ldots)$, $\lambda_{1,s}(\by):=\lambda_1(\by_s)$ and $\omega_{1,s}(\by):=\omega_1(\by_s)$.
Our  result in this section reads as follows.
\begin{theorem}\label{thm-main3}
Let $s\in \NN$ and $N\in \NN$ be prime, $\mathcal{G}\in V^*$. Let Assumption \ref{assumption} hold  with  a non-increasing sequence $(\rho_j^{-1})_{j\in \NN}$   and $(\rho_j^{-1})_{j\in \NN}\in \ell_p(\NN)$ for $p\in (0,1]$. Then a randomly shifted lattice rule with $N$ points can be constructed in $\mathcal{O}(sN\log N)$ operations using the fast CBC algorithm such that 
\begin{equation}\label{eq-error-value}	\sqrt{\mathbb{E}^\varDelta\big(|\E_{\by}[\lambda_1]-Q_{s,N}^\varDelta(\lambda_{1,s})|^2\big)}
\leq C_\lambda\Big( \min\Big\{ \rho_{s+1}^{-1},  s^{-2(\frac{1}{p}-1)}\Big\}  +N^{-\alpha} \Big)
\end{equation}
and
\begin{equation}\label{eq-error-function}
\sqrt{\mathbb{E}^\varDelta\big(|\E_{\by}[\mathcal{G}(\omega_{1})]-Q_{s,N}^\varDelta(\mathcal{G}(\omega_{1,s}))|^2\big)} 
\leq 
C_\omega 
\begin{cases}
 s^{-2(\frac{1}{p}-1)}    + N^{-\alpha} &\text{if} \ \ p<1
 	\\[1ex]
\big(\sum_{j= s+1}^\infty \rho_j^{-1}\big)^2 +  N^{-\frac{1}{2}} & \text{if}\ \  p=1
\end{cases}
\end{equation}
where 
$$
\alpha=
\begin{cases}
	1-\delta, \text{for arbitrary }\delta\in(0,\frac{1}{2}), & \text{if } p\in (0,\frac{2}{3}]
	\\
	\frac{1}{p}-\frac{1}{2}& \text{if } p\in (\frac{2}{3},1]
\end{cases}
$$
and the positive constants $C_\lambda$ and $C_\omega$ are independent of $s$ and $N$. 
\end{theorem}

Before going to proof, we need a truncation estimation. 
Using the analyticity of $\lambda_1(\by)$ and $\omega_1(\by)$ we can prove the following. 
\begin{lemma}\label{lem:dim}
	Let $s\in \NN$ and Assumption \ref{assumption} hold. 	Let $\Gg\in V^*$. Assume that $(\rho_j^{-1})_{j\in \NN}\in \ell_p(\NN)$ with $p\in (0,1]$ and  $(\rho_j^{-1})_{j\in \NN}$ is non-increasing. If $p\in (0,1)$ then  we have
	$$ 
	\big|\E_\by\big[\lambda_1-\lambda_{1,s}\big]\big| 	\leq
	C_0^2\frac{   K_\lambda}{4\eta^2} s^{-2(\frac{1}{p}-1)} 
	$$
	and
\begin{equation}\label{eq-proved}	\big|\mathbb{E}_\by\big[\Gg(\omega_1)-\Gg(\omega_{1,s})\big]\big| 
	\leq
	C_0^2\frac{  \|\Gg\|_{V^*}K_\omega }{4\eta^2} s^{-2(\frac{1}{p}-1)},
\end{equation}
	where $C_0=\min\big(\frac{p}{1-p},1 \big)\big\|(\rho_j^{-1})_{j\in \NN}\big\|_{\ell_p(\NN)}$ and $K_\lambda$, $K_\omega
	$, $\eta $ are given in Theorem \ref{Thm:derivatives}. When $p=1$, it holds
	$$ 
\big|\E_\by\big[\lambda_1-\lambda_{1,s}\big]\big| 	\leq
 \frac{   K_\lambda}{4\eta^2} \Bigg(\sum_{j= s+1}^\infty \rho_j^{-1}\Bigg)^2  
$$
	and
	$$ \big|	\mathbb{E}_\by\big[\Gg(\omega_1)-\Gg(\omega_{1,s})\big]  \big|
	\leq  \frac{ \|\Gg\|_{V^*}K_\omega}{4\eta ^2} \Bigg(\sum_{j= s+1}^\infty \rho_j^{-1}\Bigg)^2 \,.
	$$
\end{lemma}
\begin{proof}
We prove \eqref{eq-proved}. The other bounds are carried out  similarly. First, we recall Stechkin's estimate
	\be \label{Stechkin}
	\sum_{j= s+1}^\infty \rho_j^{-1} \leq \min\Big(\frac{p}{1-p},1 \Big)\big\|(\rho_j^{-1})_{j\in \NN}\big\|_{\ell_p(\NN)}s^{-(\frac{1}{p}-1)}\,,
	\ee 
	see \cite[Theorem 5.1]{Kuo12}. Since $\omega_1(\by)$ is analytic, see Theorem \ref{thm:analytic}, by Taylor's Theorem with integral form of the remainder we obtain
	\beqq
	\omega_1(\by)=\omega_1(\by_s)+\sum_{j=s+1}^\infty y_j\int_0^1 \partial^{\bee_j}\omega_1\big(\by_s+t\by_{\bar{s}}\big)\rd t\,,
	\eeqq
	where $\bee_j=(\delta_{j,\ell})_{\ell\in \NN}$ and $\delta_{j,\ell}$ denotes the Kronecker delta.	Employing Theorem \ref{Thm:derivatives} with the fact that $|y_j|\leq \frac{1}{2}$ yields
	\begin{align*}
		\big\|\omega_1(\by)-\omega_{1,s}(\by)\big\|_V & \leq \frac{1}{2}\sum_{j= s+1}^\infty\int_0^1\big\| \partial^{\bee_j}\omega_1\big(\by_s+t\by_{\bar{s}}\big)\big\|_V\rd t
		\\
		&	\leq \frac{1}{2} \sum_{j= s+1}^\infty \frac{K_{\omega}}{\eta } \rho_j^{-1}  \leq \frac{K_\omega}{2\eta } C_0 s^{-(\frac{1}{p}-1)}\,,
	\end{align*}
	where in the last inequality we used \eqref{Stechkin}. It has been proved in \cite[Theorem 4.1]{Gil18} that
	\begin{equation*}
		\mathbb{E}_\by\big[\Gg(\omega_1)-\Gg(\omega_{1,s})\big] = \sum_{\ell,j= s+1}^\infty \mathbb{E}_\by\Bigg[\Gg \bigg(\frac{2y_\ell y_j}{(\bee_\ell +\bee_j)!}\int_0^1(1-t)\partial^{\bee_\ell +\bee_j}\omega_1\big( \by_s+t\by_{\bar{s}}\big)\rd t \bigg) \Bigg]\,.
	\end{equation*}
	By the linearity of $\Gg$, from Theorem \ref{Thm:derivatives}  we obtain 
	\begin{equation*}
		\begin{split}
			\Bigg|\Gg \bigg(\frac{2y_\ell y_j}{(\bee_\ell +\bee_j)!}\int_0^1(1-t)\partial^{\bee_\ell +\bee_j}\omega_1\big( \by_s+t\by_{\bar{s}}\big)\rd t \bigg) \Bigg|
			& \leq 
			\frac{1}{4}\|\Gg\|_{V^*} \frac{\big\| \partial^{\bee_\ell +\bee_j}\omega_1\big( \by_s+t\by_{\bar{s}}\big)\big\|_V}{(\bee_\ell +\bee_j)!}
			\\
			& \leq \frac{1}{4} \|\Gg\|_{V^*} \frac{K_\omega}{\eta ^2}\rho_\ell ^{-1}\rho_j^{-1}\,.
		\end{split} 
	\end{equation*}
	This leads to
	\begin{equation*} 
		\begin{split} 
			\big|	\mathbb{E}_\by\big[\Gg(\omega_1)-\Gg(\omega_{1,s})\big]   \big|
			& \leq \frac{ \|\Gg\|_{V^*}K_\omega}{4\eta ^2} \sum_{\ell,j= s+1}^\infty \rho_\ell ^{-1}\rho_j^{-1} 
			\\
			& = \frac{ \|\Gg\|_{V^*}K_\omega}{4\eta ^2} \Bigg(\sum_{j= s+1}^\infty \rho_j^{-1}\Bigg)^2 \leq \frac{  \|\Gg\|_{V^*}K_\omega }{4\eta ^2} C_0^2s^{-2(\frac{1}{p}-1)}\,.
		\end{split}
	\end{equation*}
	The proof is finished.
	\hfill
\end{proof}

We also have the following.
\begin{lemma}\label{lem:dim-lambda}
Let $s\in \NN$.	Under Assumption \ref{assumption}, we have
	\begin{align*}	
		\big|\E_\by\big[\lambda_1 -\lambda_{1,s} \big]\big|   
		\leq C_1 \frac{\Lambda_1}{2}
		\sup_{j\geq s+1}\rho_j^{-1},
	\end{align*}
	where   $C_1$ is  given in \eqref{eq-continuos}. 
\end{lemma}
\begin{proof} 
From \eqref{eq-continuos} we have
	\begin{align*}
		|\lambda_1(\by)-\lambda_{1,s}(\by)| 
		&
		\leq C_1
		\big\|a(\by)-a(\by_s)\big\|_{L_\infty(D)}  =C_1 \Bigg\| \sum_{j\geq s+1}y_j a_j\Bigg\|_{L_\infty(D)}
		\\
		&
		\leq 
		\frac{C_1}{2}   \sup_{j\geq s+1}\rho_j^{-1}   \Bigg\|\sum_{j\geq s+1}|\rho_ja_j|\Bigg\|_{L_\infty(D)} 
		\leq C_1 \frac{\Lambda_1}{2} \sup_{j\geq s+1}\rho_j^{-1}\,.
	\end{align*} 
This leads to the desired result.
	\hfill
\end{proof}
\begin{remark}
	By using estimates in Remark \ref{rem:global} we obtain a similar result as \cite[Theorem 4.1]{Gil18} for $p\in (0,1]$. Note that the constant $C_{\text{trunc}}$ in \cite[Theorem 4.1]{Gil18} depends on the other constant $C_\varepsilon$ and $C_{\text{trunc}}$ blows up when $\varepsilon$ goes to zero.
\end{remark}
We are now in the position to prove Theorem \ref{thm-main3}.

\begin{proof} We prove the error bound \eqref{eq-error-value}. The error bound \eqref{eq-error-function} is carried out similarly. Using triangle inequality we get
\begin{equation*}\label{eq-split}
\begin{aligned}
&\sqrt{\mathbb{E}^\varDelta
	\big(|\E_{\by}[\lambda_1]-Q_{s,N}^\varDelta(\lambda_{1,s})|^2\big)}
\\
&  \leq C\bigg(\sqrt{\mathbb{E}^\varDelta\big(|\E_{\by}[\lambda_1]-\E_{\by}[\lambda_{1,s}]|^2\big)} + \sqrt{ \mathbb{E}^\varDelta\big(|\E_{\by}[\lambda_{1,s}]-Q_{s,N}^\varDelta(\lambda_{1,s})|^2\big)}\,\bigg).
\end{aligned}
\end{equation*}
Since the first term on the right-hand side is independent of the random shift, by Lemmas \ref{lem:dim-lambda} and \ref{lem:dim} we get
$$
\sqrt{\mathbb{E}^\varDelta\big(|\E_{\by}[\lambda_1]-\E_{\by}[\lambda_{1,s}]|^2\big)} = \sqrt{ |\E_{\by}[\lambda_1]-\E_{\by}[\lambda_{1,s}]|^2}\leq  \min\Big\{ \rho_{s+1}^{-1},  s^{-2(\frac{1}{p}-1)}\Big\}.
$$
 For $\mfu \subseteq [s]$ we put $\bnu=(\nu_j)_{j\in \NN}\in \FF$ with $\nu_j\leq 1$ and $\supp(\bnu)=\mfu$.
From Theorem \ref{Thm:derivatives} we get
$$
\bigg|\frac{\partial ^{|\mfu|}
	\lambda_{1,s}}{\partial \bxi_{\mfu}}\big(\bxi\big)\bigg| \leq K_{\lambda}
\frac{\bnu! } {(\eta \brho)^\bnu}= K_\lambda \prod_{j\in \mfu}\frac{1}{\eta\rho_j}, \quad \bxi=(\xi_1,\ldots,\xi_s),
$$
which implies
\begin{align*}
	\|\lambda_{1,s}\|_{\Ww_\bgamma(U^s)}^2\leq  \sum_{\mfu\subseteq [s]}\frac{1}{\gamma_{\mfu}} \int_{U^{|\mfu|}}\Bigg(\int_{U^{|\bar{\mfu}|}}	K_\lambda \prod_{j\in \mfu}\frac{1}{\eta\rho_j}\rd \bxi_{\bar{\mfu}} \Bigg)^2\rd\bxi_{\mfu}
	\leq \sum_{\mfu\subseteq [s]}\frac{1}{\gamma_{\mfu}}  \Bigg( K_{\lambda}
\prod_{j\in \mfu}\frac{1}{\eta\rho_j}  \Bigg)^2 .
\end{align*}
By Proposition \ref{QMC}  we find
\begin{align*}
\sqrt{\mathbb{E}^\varDelta\big(|\E_{\by}[\lambda_{1,s}]-Q_{s,N}^\varDelta(\lambda_{1,s})|^2\big)}
	\leq 
	K_{\lambda}\Bigg( \sum_{\mfu \subseteq [s] }\gamma_{\mfu}^\lambda \bigg(\frac{2\zeta(2\lambda)}{(2\pi^2)^\lambda} \bigg)^{|\mfu|}\Bigg)^{\frac{1}{2\lambda}} 	 \Bigg(\sum_{\mfu\subseteq [s]}\frac{1}{\gamma_{\mfu}}  \prod_{j\in \mfu}\frac{1}{(\eta\rho_j)^2} \Bigg)^{\frac{1}{2}} N^{-\frac{1}{2\lambda}},
\end{align*}
Choosing 
$$\gamma_j=\bigg(\frac{1}{(\eta\rho_j)^2}:\frac{2\zeta(2\lambda)}{(2\pi^2)^\lambda}\bigg)^{\frac{1}{1+\lambda}}, \ j\in \mfu$$
we get
\begin{align*}
\sqrt{\mathbb{E}^\varDelta\big(|\E_{\by}[\lambda_{1,s}]-Q_{s,N}^\varDelta(\lambda_{1,s})|^2\big)}
&	\leq 
	C_{\lambda}\Bigg( \sum_{\mfu \subseteq [s]}\prod_{j\in \mfu} \big(\eta \rho_j\big)^{\frac{-2\lambda}{1+\lambda}}\bigg(\frac{2\zeta(2\lambda)}{(2\pi^2)^\lambda}\bigg)^{\frac{1}{1+\lambda}}\Bigg)^{\frac{1+\lambda}{2\lambda}} N^{-\frac{1}{2\lambda}}
	\\
	&\leq C_\lambda \exp\Bigg(\frac{1+\lambda}{2\lambda}\bigg(\frac{2\zeta(2\lambda)}{(2\pi^2)^\lambda}\bigg)^{\frac{1}{1+\lambda}}\sum_{j\in \NN} \big(\eta \rho_j\big)^{\frac{-2\lambda}{1+\lambda}}\Bigg)N^{-\frac{1}{2\lambda}},
\end{align*}
where we have applied \cite[Lemma 6.3]{Kuo12} in the last step in the second inequality. We consider two cases. If $p\in (0,\frac{2}{3}]$, for $\delta\in (0,\frac{1}{2})$ we choose $\lambda=\frac{1}{2(1-\delta)}$. By this choice it is easily seen that $\frac{2\lambda}{1+\lambda} \geq \frac{2}{3} $. Therefore $\sum_{j\in \NN} \big(\eta \rho_j\big)^{\frac{-2\lambda}{1+\lambda}}<\infty$ since $(\rho_j^{-1})_{j\in \NN}\in \ell_p(\NN)$. If $p\in (\frac{2}{3},1]$ we choose $\lambda$ such that $\frac{2\lambda}{1+\lambda}=p$ which implies $\frac{1}{2\lambda}=\frac{1}{p}-\frac{1}{2}$.  The proof is completed.
\hfill
\end{proof}

\providecommand{\bysame}{\leavevmode\hbox to3em{\hrulefill}\thinspace}
\providecommand{\MR}{\relax\ifhmode\unskip\space\fi MR }
\providecommand{\MRhref}[2]{%
	\href{http://www.ams.org/mathscinet-getitem?mr=#1}{#2}
}
\providecommand{\href}[2]{#2}

\end{document}